\documentclass{article}
\usepackage[utf8]{inputenc}
\usepackage{amsmath}
\usepackage{amsfonts, amssymb,amsthm}

\newtheorem{thm}{Theorem}
\newtheorem{lem}[thm]{Lemma}

\newtheorem{rem}[thm]{Remark}

\def\ed#1{ {\mathbf 1}_{ \{#1  \}}}             % indicator

\def\AA{{\mathcal A}}
\def\BB{{\mathcal B}}

\def\E{{\mathbb E}\,}

\def\JJ{{\mathcal J}}
\def\MM{{\mathcal M}}
\def\N{{\mathbb N}}

\def\P{{\mathbb P}}

\def\R{{\mathbb R}}
\def\SS{{\mathcal S}}

\def\Var{\textrm{Var}\,}
\def\tY{{\widetilde Y}}

\title{Extrema of multinomial assignment process}

\author{
Mikhail Lifshits\footnote{St.Petersburg State University. Russia, 191023, St.Petersburg, University Emb. 7/9. \texttt{mikhail@lifshits.org}} 
\ and 
Gilles Mordant\footnote{Institut fur Mathematische Stochastik, Universit\"at G\"ottingen.
G\"ottingen, Germany. \texttt{gilles.mordant@uni-goettingen.de}}
}
\begin{document}

\maketitle

\begin{abstract}
 We study the asymptotic behavior of the expectation of the maxima and minima of random assignment process generated by a large matrix with multinomial entries. A variety of results is obtained for different sparsity regimes.
\end{abstract}
\medskip

\medskip


\section{Introduction and main results}

\subsection{Random assignment problem}
We consider the following \textit{random assignment problem}.  Let ($X_{ij}$) be an $n\times n$ random matrix and let $[1..n]$ denote the set $\{1,2, \ldots, n\}$.
Let $\SS_n$ denote the group of permutations
$\sigma : [1..n] \mapsto [1..n]$. For every $\sigma\in \SS_n$, let
\[
   S(\sigma)=\sum\limits_{i=1}^{n} X_{i\sigma(i)}.
\]

The process $\{S(\sigma),\, \sigma \in \SS_n\}$ is called a \textit{random assignment process}.
The problem consists in the study of the asymptotic behaviour of its extrema, in particular,
\begin{equation} \label{minmax}
   \E \max\limits_{\sigma\in \SS_n} S(\sigma) \qquad \textrm{and } \qquad
   \E \min\limits_{\sigma\in \SS_n} S(\sigma),  \qquad \textrm{as } n \to \infty.
\end{equation}
We refer to \cite{ CoppersmithSorkin,Steele}
for many applications of assignment processes and their extrema in various fields of mathematics.
\medskip

There are many remarkable results in the area, including a famous result of Aldous  \cite{AldousZet2}
who proved a conjecture  by  M\'ezard and Parisi claiming that
\[
   \lim_{n\to\infty} \E \min\limits_{\sigma\in \SS_n} S(\sigma) = \frac{\pi^2}{6}
\]
when the $X_{ij}$ are i.i.d.\ standard exponential. Actually, he showed that, when the random variables considered are nonnegative, the distribution of
$X_{ij}$ affects the limit in the minimisation problem only through the value of its probability density function at 0.

In the mentioned case, the common distribution is bounded from below. The situation is very different when one deals with the variables having unbounded
distributions.  For obvious reasons,  it is more convenient to illustrate this phenomenon for maxima instead of minima. If the common law of the entries is not bounded from above, then the expectation of maxima does not tend anymore to a finite limit but grows to infinity and the problem consists in evaluation of the
corresponding growth order. In this direction, Mordant and Segers \cite{MS} showed that
if  $X_{ij}$ are i.i.d. standard Gaussian, then
\[
    \E{\max_{\sigma\in \SS_n} S(\sigma)} = n \sqrt{2 \log{n}} (1 + o(1)).
\]
Some rather general results of this type were recently obtaind by Cheng et al. \cite{TTkocz} and
Lifshits and Tadevosian \cite{LT}.

Not so much is known for the assignment problem in the discrete setting. One may mention the case
of i.i.d. Poisson random variables studied in \cite{LT} and a work of Parviainen \cite{Parv}
who considered uniform distributions on $[1..n]$, or on $[1..n^2]$, random permutations
of $[1..n]$ for each row, and those of $[1..n^2]$ for the whole matrix.

In this article, we study \eqref{minmax} for random matrices $X=(X_{ij})_{1\le i,j\le n}$ with the joint {\it multinomial}
distribution of entries $\MM(m,n^2)$. Therefore, the matrix entries are integer-valued, negatively dependent random
variables with common  binomial distribution $\BB(m,p)$ with success probability $p=n^{-2}$ and number of trials $m$.
We allow the dependence $m=m(n)$. As one will see, the presence of this extra parameter $m$ creates a space
for a variety of asymptotic behaviors for the expectation of the extrema.

\subsection{A motivating example} Let us give an example showing how the studied problem emerges in
information transmission. Let
$\AA=(a_1, ...,a_n)$ be an alphabet of $n$ letters. If $u$ and $v$ are two
independent uniformly distributed words of length $m$, the $n\times n$ matrix $X$ defined by
\[
   X_{ij} := \sum_{k=1}^m \ed{u_k=a_i,v_k=a_j}, \qquad 1\le i,j\le n,
\]
is distributed according to the multinomial law $\MM(m,n^2)$. Recall that  Hamming distance between the words is defined by
\[
   d_H(u,v) := \sum_{k=1}^m \ed{u_k \not =v_k} = m- \sum_{k=1}^m \ed{u_k =v_k}
   = m - \sum_{i=1}^n X_{ii}.
\]

Assume that we have received a word $v$ through a noisy channel and we have to decide whether $v$
is just a random word or a word $u$ that passed through an unknown coding $\sigma:\AA\mapsto\AA$.
The answer should clearly depend on the quantity
\[
     \min_\sigma d_H(\sigma(u),v) =  \min_\sigma \left( m - \sum_{i=1}^{n} X_{i\sigma(i)}\right)
     =m - \max_\sigma \sum_{i=1}^{n} X_{i\sigma(i)}.
\]

\subsection{Results}

Our setting is an asymptotic one, i.e., we let $n\to\infty$ and allow $m=m_n$ to be a function of $n$.
The results depend heavily on the relation between $n$ and $m$.
Therefore, we consider separately several zones gradually going down from large $m$'s to the smaller ones.
Everywhere we use the notation  $p=p_n:=n^{-2}$ for the probability which is naturally related to our basic
multinomial law $\MM(m,n^2)$. All limits are meant for $n\to \infty$.

\subsubsection*{Quasi-Gaussian zone}

This zone is defined by assumption
\begin{equation} \label{subgaussian}
   \frac{mp}{\log n} \to \infty
\end{equation}
which essentially means that all entries $X_{ij}$ are sufficiently large to be heuristically approximated
with Gaussian variables.

\begin{thm} \label{t:subgaussian}
  Under assumption \eqref{subgaussian} it is true that
\begin{eqnarray*}
   \E \max_\sigma \sum_{i=1}^n X_{i\sigma(i)} &=& \frac{m}{n} \, (1+o(1)),
\\
    \E \min_\sigma \sum_{i=1}^n X_{i\sigma(i)} &=& \frac{m}{n} \, (1+o(1)).
\end{eqnarray*}
\end{thm}

\subsubsection*{Critical zone}

The critical zone is described by assumption
\begin{equation} \label{critical}
    \frac{m p}{\log n} \to c
\end{equation}
with some $c>0$. Unlike to the
 quasi-Gaussian case, the expectation behavior of maxima and minima
is not the same anymore.

\begin{thm} \label{t:critical}
  Under assumption \eqref{critical}  for all $c>0$ it is true that
\[
     \E \max_\sigma \sum_{i=1}^n X_{i\sigma(i)} = c \, H_* n \log n \, (1+o(1)),
\]
where $H_*=H_*(c)$ is the unique solution of equation
\begin{equation} \label{Hstar}
    \begin{cases}
    H \log H- (H-1) = \frac 1c, &\\
    H>1, \\
    \end{cases}
\end{equation}
and for  all $c>1$ it is true that
\[
    \E \min_\sigma \sum_{i=1}^n X_{i\sigma(i)} =  c\, \widetilde H_* n \log n \, (1+o(1)),
\]
where $\widetilde{H}_*=\widetilde{H}_*(c)$ is the unique solution of equation
\begin{equation} \label{tHstar}
    \begin{cases}
    H \log H- (H-1) = \frac 1c, &\\
    0<H<1. \\
    \end{cases}
\end{equation}
\end{thm}

For $c<1$ equation \eqref{tHstar} has no solution and the result for the minimum is completely different,
as  stated in the next theorem.

\begin{thm}  \label{t:ER}  Let $c<1$ and
\begin{equation} \label{alpha}
     \limsup \frac{mp}{\log n} \le c.
\end{equation}
Then,
\[
   \lim \P\left( \min_\sigma \sum_{i=1}^n X_{i\sigma(i)} =0 \right) =1.
\]
\end{thm}

\begin{rem} {\rm
 The intermediate case $c=1$ admits a similar treatment but the result is
 less attractive. For example, one may replace assumption \eqref{alpha} with
\[
    \frac{mp}{\log n} \le 1 - \frac{\log(b\log n)}{\log n}, \qquad b>1.
\]
 }
\end{rem}

\subsubsection*{Quasi-Poissonian zone}

The quasi-Poissonian zone is described by the assumptions
\begin{equation} \label{Poisson}
    \frac{m p}{\log n} \to 0
\end{equation}
while, for every $\delta>0$,
\begin{equation} \label{delta}
  mp\gg n^{-\delta}.
\end{equation}
In this zone all entries $X_{ij}$ are well approximated by
Poissonian variables with intensity parameter $mp$.
This zone includes moderately growing intensities $mp$, the constant $mp$ and even a narrow zone
of $mp$ slowly decreasing to zero, e.g., with logarithmic speed.

\begin{thm} \label{t:Poisson}
Under assumptions \eqref{Poisson} and \eqref{delta}
it is true that
\begin{equation} %% \label{Emax_Poisson}
  \E \max_\sigma \sum_{i=i}^{n} X_{i\sigma(i)}
  =
    \frac{n\, \log n}{\log\left(\frac{\log n}{mp}\right)}\, (1+o(1)).
\end{equation}
\end{thm}

\begin{rem} {\rm
   Note that if $\log (mp) \ll \log\log n$ we obtain asymptotics $\tfrac{n\, \log n}{\log\log n}$
   as in the Poisson i.i.d. case with constant intensity \cite{LT}.
   }
\end{rem}

\subsubsection*{Rather sparse matrices}

In this zone, we go below \eqref{delta} and assume that
\begin{equation} \label{neg_polynomial}
    mp = c \, n^{-a} \, (1+o(1)), \qquad  a\in (0,1).
\end{equation}
Consider first a regular case.

\begin{thm} \label{t:Rsparse}
Assume that \eqref{neg_polynomial} holds and
\begin{equation} \label{regular_a}
   a \not \in \left\{ \frac 1k,\ k\in \N \right\}.
\end{equation}
Then, there exists a unique positive integer $k$ such that
\begin{equation} \label{bounds_a}
  \frac{1}{k+1} < a < \frac{1}{k}
\end{equation}
and
\begin{equation} \label{Emax_a}
      \E \max_\sigma \sum_{i=1}^n X_{i\sigma(i)}  = k\, n\, (1+o(1)).
\end{equation}
\end{thm}

Let us now briefly discuss the irregular case $a= \tfrac{1}{k}$ for some integer
$k\ge 2$.  Since the lower bound $a>\tfrac 1{k+1}$ is still true,
one may obtain again
\[
    \E \max_\sigma \sum_{i=1}^n X_{i\sigma(i)}  \le k \,  n\, (1+o(1)).
\]
However, the opposite bound breaks down and we are only able to prove that
\[
   \E  \max_\sigma \sum_{i=1}^n X_{i\sigma(i)}  \ge (k-1)\, n (1+o(1)).
\]
%%If we want to get something better, in the greedy method we have to apply the lower bound with the varying number of variables $v=n,n-1,n-2,...$
%% which  brings a correction constant due to the presence of  $v$ in the right hand side of the bound \eqref{bound_v}.
To summarise, for the assignment process, we have in this case that
\[
  (k-1)\, n\, (1+o(1)) \le  \E \max_\sigma \sum_{i=1}^n X_{i\sigma(i)}  \le k \,n\, (1+o(1))
\]
and conjecture that
\[
   \E \max_\sigma \sum_{i=1}^n X_{i\sigma(i)} = (k-\kappa)\, n (1+o(1)),
\]
for some $\kappa\in [0,1]$ depending on $a$ and $c$.  Proving this and finding $\kappa$ is beyond the reach of current techniques.

\subsubsection*{Very sparse matrices}
This zone is determined by
\begin{equation} \label{Vsparse}
   1 \ll m \ll n.
\end{equation}

Notice that $m\approx n$ is equivalent to $mp\approx n^{-1}$, thus the current zone is just below the previous one.

\begin{thm} \label{t:Vsparse} Under assumption \eqref{Vsparse} it is true that
\[
    \E \max_\sigma \sum_{i=1}^n X_{i\sigma(i)} = m \, (1+o(1)).
\]
\end{thm}

%%%%%%%%%%%%%%%%%%%%%%%%%%%%%%%%%%%%%%%%%%%%%%%%%%%%%%%%%%%%%%%%%%%%%%%%%%%%%%%%%
%%%%%%%%%%%%%%%%%%%%%%%%%%%%%%%%%%%%%%%%%%%%%%%%%%%%%%%%%%%%%%%%%%%%%%%%%%%%%%%%%

\section{Proofs}

 \begin{proof}[Proof of Theorem \ref{t:subgaussian}]
Let $X$ be a  $\BB(m,p)$-distributed random variable. Then,
\begin{equation}\label{Eexp}
   \E \exp(\gamma X) = (1+p(e^\gamma-1))^m, \qquad \gamma\in \R.
\end{equation}
Let now  $X_j, 1\le j\le n$, be $\BB(m,p)$-distributed random variables.
We do not assume any independence. Then, for every $\gamma>0$, we have
\[
    \E \exp(\gamma \max_{1\le j\le n} X_j)
    \le \E \sum_{j=1}^n \exp(\gamma X_j) = n \, (1+p(e^\gamma-1))^m.
\]
By Jensen inequality,
\[
  \exp \left( \gamma\, \E \max_{1\le j\le n} X_j \right)
  \le \E \exp(\gamma \max_{1\le j\le n} X_j)
  \le n \, (1+p(e^\gamma-1))^m.
\]
It follows that
\begin{eqnarray*}
    \E \max_{1\le j\le n} X_j
    &\le& \gamma^{-1} \left( \log n + m \log(1+p(e^\gamma-1)) \right)
\\
    &\le& \gamma^{-1} \left( \log n + mp (e^\gamma-1) \right).
\end{eqnarray*}

We choose  $\gamma:=(\tfrac{2\log n}{mp})^{1/2}$.
By \eqref{subgaussian} we have $\gamma\to 0$. Using the expansion
$e^\gamma-1= \gamma +\gamma^2(1+o(1))/2$, we obtain
\begin{eqnarray*}
  \E \max_{1\le j\le n} X_j
  &\le& \gamma^{-1} \left( \log n + mp [\gamma +\gamma^2(1+o(1))/2] \right)
\\
  &=& mp + \gamma^{-1} \log n + mp\,\gamma(1+o(1))/2
\\
   &=& mp + (2 mp \log n)^{1/2} (1+o(1)).
\end{eqnarray*}
Furthermore, by \eqref{subgaussian} the second term is negligible and we obtain
\[
    \E \max_{1\le j\le n} X_j \le m \, p \, (1+o(1)).
\]

The same approach applies to the minima. With the same notation we have
for every $\gamma>0$
\[
    \E \exp(- \gamma \min_{1\le j\le n} X_j)
    \le \E \sum_{j=1}^n \exp(-\gamma X_j) = n \, (1+p(e^{-\gamma}-1))^m.
\]
By Jensen inequality,
\[
  \exp \left( - \gamma\, \E \min_{1\le j\le n} X_j \right)
  \le \E \exp(- \gamma \min_{1\le j\le n} X_j)
  \le n \, (1+p(e^{-\gamma}-1))^m.
\]
It follows that
\begin{eqnarray*}
    \E \min_{1\le j\le n} X_j
    &\ge&  - \gamma^{-1} \left( \log n + m \log(1+p(e^{-\gamma}-1)) \right).
%%\\
%%    &\le&  -\gamma^{-1} \left( \log n + mp (e^{-\gamma}-1) \right).
\end{eqnarray*}
We still use $\gamma:=(\tfrac{2\log n}{mp})^{1/2}\to 0$.
The expansion $e^{-\gamma}-1= -\gamma +\gamma^2(1+o(1))/2$ yields
\[
  \log(1+p(e^{-\gamma}-1))  =  p (e^{-\gamma}-1) (1+o(1))
  =  - p \gamma (1+o(1))  + p \gamma^2 (1+o(1)) / 2.
\]
From this we get
\begin{eqnarray*}
  \E \min_{1\le j\le n} X_j
  &\ge& -\gamma^{-1} \left( \log n + mp [-\gamma (1+o(1))  +\gamma^2(1+o(1))/2] \right)
\\
  &=& mp  (1+o(1)) - \gamma^{-1} \log n - mp\,\gamma(1+o(1))/2
\\
   &=& mp (1+o(1))  - (2 mp \log n)^{1/2} (1+o(1)).
\end{eqnarray*}
By \eqref{subgaussian} the second term is negligible and we obtain
\[
    \E \min_{1\le j\le n} X_j \ge mp \, (1+o(1)).
\]

Let us now apply these results to the multinomial assignment process.
Here the joint law of the entries $X_{ij}$ is $\MM(m, n^{2})$ and every
$X_{ij}$ follows Binomial law $\BB(m,p)$ with $p=n^{-2}$.
Our bound for the maxima yields
\[
   \E \max_\sigma \sum_{i=1}^n X_{i\sigma(i)}
   \le \sum_{i=1}^n \E \max_{1\le j\le n} X_{ij}
   = n \cdot  \E \max_{1\le j\le n} X_{1j}
   \le \frac{m}{n} \, (1+o(1)),
\]
while the bound for the minima yields
\[
   \E \min_\sigma \sum_{i=1}^n X_{i\sigma(i)}
   \ge \sum_{i=1}^n \E \min_{1\le j\le n} X_{ij}
   = n \cdot  \E \min_{1\le j\le n} X_{1j}
   \ge \frac{m}{n} \, (1+o(1)).
\]

It follows that
\begin{eqnarray*}
   \E \max_\sigma \sum_{i=1}^n X_{i\sigma(i)} &=& \frac{m}{n} \, (1+o(1)),
\\
    \E \min_\sigma \sum_{i=1}^n X_{i\sigma(i)} &=& \frac{m}{n} \, (1+o(1)),
\end{eqnarray*}
as required.
\end{proof}
\medskip

%%%%%%%%%%%%%%%%%%%%%%%%%%%%%%%%%%%%%%%%%%%%%%%

 \begin{proof}[Proof  of Theorem \ref{t:critical}]
Let $(X_j)$ be negatively associated random variables following the Bernoulli law $\BB(m,p)$.
We claim that for every $c>0$  under \eqref{critical}  and under the additional assumption
\begin{equation}  \label{plogn}
  p \, \log n \to 0,
\end{equation}
it is true that
\begin{equation}  \label{Emax_critical}
  \E \max_{1\le j\le n} X_j = c H_* \log n \, (1+o(1)),
  \qquad \textrm{as } n\to\infty.
\end{equation}
Further, for every $c>1$,
\begin{equation} \label{Emin_critical}
  \E \min_{1\le j\le n} X_j = c \widetilde{H}_* \log n \, (1+o(1)),
  \qquad \textrm{as } n\to\infty.
\end{equation}
\medskip

{\bf The upper bound in \eqref{Emax_critical} and the lower bound in \eqref{Emin_critical}}.
Let $H>H_*$. Then
\begin{equation}  \label{Hc}
   H\log H-(H-1)>\frac 1c.
\end{equation}
Let $r:=Hp$. Then, by \eqref{critical}, $m r=H m p = c\, H\, \log n \, (1+o(1))$.

 Applying the exponential Chebyshev inequality for every $j$ and every $v>0$,  we obtain
\begin{eqnarray} \nonumber
  \P(X_j\ge c H \log n+ v) &=& \P(X_j\ge mr +v)
\\  \label{Cheb}
  &\le& \frac{\E e^{\gamma X_j}}{e^{\gamma (m r+v)}}
  = \left[ \frac{1+p(e^\gamma-1)}{e^{\gamma r}} \right]^m \ e^{-\gamma v}.
\end{eqnarray}
By choosing the optimal $\gamma:= \log\left(\frac{(1-p)r}{p(1-r)}\right)$, we have
\begin{eqnarray*}
 \frac{1+p(e^\gamma-1)}{e^{\gamma r}} &=& \left( \frac pr\right)^r
 \exp\left((1-r)\log(1-p)- (1-r)\log(1-r) \right)
\\
     &=& H^{-Hp}  \exp\left(-p+r + O(p^2) \right)
\\
     &=&  \exp\left( - (H\log H -(H-1)) p + O(p^2) \right).
\end{eqnarray*}
Hence,
\begin{eqnarray*}
   \left[
  \frac{1+p(e^\gamma-1)}{e^{\gamma r}}\right]^m
  &=& \exp\left( - (H\log H -(H-1) + o(1))\, mp  \right)
\\
  &=&  \exp\left( - (H\log H -(H-1))\, c\, \log n (1+o(1)) \right)
\\
  &:=& n^{-\beta+o(1)},
\end{eqnarray*}
where by \eqref{Hc}  it is true that
\[
   \beta= (H\log H -(H-1)) c > 1.
\]

Substituting the above results in \eqref{Cheb} we obtain
\[
   \P(X_j\ge c H \log n +v) \le  n^{-\beta+o(1)} \ e^{-\gamma v}.
\]
It is now trivial that
\[
   \P(  \max_{1\le j\le n} X_j \ge c H \log n +v) \le  n^{-(\beta-1)+o(1)} \ e^{-\gamma v}.
\]
It follows that
\begin{eqnarray*}
  \E \max_{1\le j\le n} X_j - cH\log n
  &=&  \E \left( \max_{1\le j\le n} X_j - cH\log n\right)
\\
  &\le&   \E \left( \max_{1\le j\le n} X_j - cH\log n\right)_+
\\
  &=& \int_0^\infty  \P(\max_{1\le j\le n} X_j \ge c H \log n +v)\, dv
\\
   &\le&   n^{-(\beta-1)+o(1)} \int_0^\infty e^{-\gamma v} \, dv
   =    n^{-(\beta-1)+o(1)}  \frac 1\gamma
\\
   &=& n^{-(\beta-1)+o(1)}  \frac {1}{\log H} (1+o(1)) \to 0.
\end{eqnarray*}
Therefore,
\[
   \E \max_{1\le j\le n} X_j \le c H\log n + o(1).
\]
By letting $H\searrow H_*$ we obtain the upper bound in \eqref{Emax_critical}.

The lower bound in \eqref{Emin_critical} is obtained in exactly the same way through
the Chebyshev inequality for the lower tails.
\medskip

{\bf Converse bounds.}
The lower bound in \eqref{Emax_critical} is reached  in a few steps.
We give a Poissonian approximation of Binomial laws,
then provide a lower bound for this Poissonian approximation.
This bound provides a lower bound for the maximum's expectation
of {\it independent} Binomial i.i.d. random variables. Finally, using negative association argument,
 we reduce the claim to the independence case.
\medskip

Step 1. Let $X$ be a Binomial $\BB(m,p)$-distributed random variable.
Elementary calculations show that Poissonian approximation
\[
    \P(X=k) = e^{-mp} \frac{(mp)^k}{k!} \ (1+o(1)
\]
is valid if $p^2m\to 0$, $pk\to 0$, and $\tfrac{k^2}{m}\to 0$.
\medskip

Step 2. Let $c>0$ and $H>1$.
Let $k=[cH \log n +1]$ and $\lambda = c \log n (1+o(1))$.
Then an elementary evaluation of Poissonian probabilities yields
\[
  e^{-\lambda} \frac{\lambda^k}{k!}
  = n^{-\beta+o(1)}
\]
where
\begin{equation}\label{beta}
   \beta:= c (H\log H - (H-1)).
\end{equation}
Now we combine the results of the two steps. Note that with
 \eqref{critical}, \eqref{plogn} and for $k = c\,H \log n \, (1+o(1))$,
all three assumptions of Step 1 are verified and, with $\lambda=mp$, we obtain
\[
  \P(X\ge cH \log n) \ge \P(X=k) =  n^{-\beta+o(1)}.
\]
If $1<H<H_*$, then $\beta<1$.
\medskip

Step 3. Let $(\widetilde{X}_j)_{1\le j\le n}$ be independent copies of $X$. Then
\begin{eqnarray} \nonumber
    \P(\max_{1\le j\le n} \widetilde X_j \le c\, H \log n)
    &=& \P(X\le c\, H \log n)^n
    \le (1- n^{-\beta+o(1)})^n
\\ \label{step3_a}
   &\le& \exp(-n^{1-\beta+o(1)}) \to 0.
\end{eqnarray}
It follows that
\begin{equation} \label{step3_b}
   \E \max_{1\le j\le n} \widetilde X_j
   \ge  cH \log n (1+o(1)).
\end{equation}
\medskip

Step 4.  From the desintegration theorem for negatively associated variables,
due to Christofides and Vaggelatou \cite{CV}, see also Bulinski and Shashkin \cite[Chapter~2,Theorem~2.6 and Lemma~2.2]{BulSh},
one has
\begin{equation} \label{step4}
    \E \max_{1\le j\le n} X_j  \ge
     \E \max_{1\le j\le n} \widetilde X_j.
\end{equation}
Combining this estimate with the result of Step 3, for every $H<H_*$ we obtain
\[
    \E \max_{1\le j\le n} X_j  \ge   cH \log n (1+o(1)).
\]
Letting $H\nearrow H_*$, we obtain the lower bound in \eqref{Emax_critical}.
\bigskip

The upper bound in \eqref{Emin_critical}
follows in a similar way.
Let now $k:=[cH \log n]$. By using Poissonian approximation and Poissonian asymptotics
we obtain
\[
   \P(X\le cH \log n) \ge \P(X=k) =
   n^{-\beta+o(1)}
\]
with the same $\beta$ from \eqref{beta}.
If $\widetilde{H}_*<H<1$, then $\beta<1$.

As before, for independent variables we obtain

\[
    \P\left(\min_{1\le j\le n} \widetilde X_j \ge cH \log n\right)
    \le \exp\big(-n^{1-\beta+o(1)}\big).
\]
It follows that
\begin{eqnarray*}
   \E \min_{1\le j\le n} \widetilde X_j
   &=&   \E \left[ \min_{1\le j\le n} \widetilde X_j \ed{ \min_{1\le j\le n} \widetilde X_j \le  cH \log n }\right]
\\
   &&       +  \E \left[ \min_{1\le j\le n} \widetilde X_j \ed{ \min_{1\le j\le n} \widetilde X_j >  cH \log n }\right]
\\
          &\le&    c\, H \log n
          +  \sum_{j=1}^n \E \left[ X_j \ed{ \min_{1\le i\le n, i\not=j} \widetilde X_i >  cH \log n }\right]
\\
          &=&   c\, H \log n
          +  n \,\E \widetilde X_1 \ \P\left( \min_{2\le i\le n} \widetilde X_i >  c\, H \log n \right)
\\
           &\le&  c\, H \log n
           +  n  \cdot c\, \log n \, (1+o(1)) \ \exp(-n^{1-\beta+o(1)})
\\
          &=&  c\, H \log n + o(1).
\end{eqnarray*}

The final negative association argument reads as follows. Since
$(X_j)$ are negatively associated, so are $(-X_j)$, too.
From  the desintegration theorem cited above it follows that
\[
   \E \max_{1\le j\le n} (-X_j) \ge \E \max_{1\le j\le n} (-\widetilde X_j)
\]
which is equivalent to
\[
   \E \min_{1\le j\le n} X_j \le \E \min_{1\le j\le n} \widetilde X_j.
\]
By combining the obtained results, we have
\[
   \E \min_{1\le j\le n} X_j \le c\, H \log n (1+o(1)).
\]
Finally, letting $H\searrow \widetilde H_*$ we obtain the upper bound
in \eqref{Emin_critical}.
\medskip

 {\bf The estimates for assignment process.}
Recall that a multinomial distribution is {\it negatively associated}, see Joag-Dev and Proschan \cite{JDP}
and Bulinski and Shashkin \cite[Chapter~1,Theorem~1.27]{BulSh}.
Furthermore, with $p=n^{-2}$, the assumption \eqref{plogn} is also valid.

Therefore, the bounds \eqref{Emax_critical} and \eqref{Emin_critical} apply
to the sums of the entries $X_{ij}$. They yield, respectively,
\begin{eqnarray*}
   \E \max_\sigma \sum_{i=1}^n X_{i\sigma(i)}
   &\le& \sum_{i=1}^{n} \E \max_{1\le j\le n} X_{ij}
   \le c\, H_* n \log n \, (1+o(1)),
\\
    \E \min_\sigma \sum_{i=1}^n X_{i\sigma(i)}
    &\ge& \sum_{i=1}^{n} \E \min_{1\le j\le n} X_{ij}
    \ge c \, \widetilde H_* n \log n \, (1+o(1)).
\end{eqnarray*}

The opposite bounds follow
by the ``greedy method'' introduced in \cite{MS} (and used in \cite{LT}) that we recall now.
This method allows to construct a quasi-optimal permutation $\sigma^*$ that provides
sufficiently large value or sufficiently small value of the assignment process.
Recall that $[1..i] := \{1, 2, \dots, i\}$.
Define
\[
    \sigma^*(1) := \arg \max\limits_{j\in[1..n]} X_{1j},
\]
and let for all  $i=2,\dots, n$
\[
    \sigma^*(i) := \arg \max\limits_{j \not\in \sigma^*([1..i-1])} X_{ij}.
\]
It is natural to call this strategy greedy, because at every step we consider the row $i$,
take the maximum of its available elements (without considering the influence of this choice on subsequent
steps) and then forget the row $i$ and the corresponding column $\sigma^*(i)$.
The number of variables used at consequent steps is decreasing from $n$ to $1$.

By using the greedy method, we have
\begin{eqnarray} \nonumber
   \E \max_\sigma \sum_{i=1}^{n} X_{i\sigma(i)}
    &\ge&  \E  \sum_{i=1}^{n} X_{i\,\sigma^*(i)}
     =  \sum_{i=1}^{n} \E \max_{ j \not\in \sigma^*([1..i-1])} X_{ij}
\\    \label{greedy}
  &=&    \sum_{i=1}^{n}   \E \max_{1\le j \le n-i+1} X_{ij}.
\end{eqnarray}
The latter equality may seem surprising because the index sets
$[n]\backslash \sigma^*([1..i-1])$ are random and depend on the matrix $X$.
However, it is justified by the following lemma.

\begin{lem} \label{l:random_sets}
   Let $N_1,N_2>0$ be positive integers and  let a random vector $X:=(X_j)_{1\le j \le N_1+N_2}$
   be distributed according to a multinomial law $\MM_{m,N_1+N_2}$. Let
   $X^{(1)}:=(X_j)_{1\le j \le N_1}$ and  $X^{(2)}:=(X_j)_{N_1< j \le N_2}$.
   Let $1\le q\le N_2$ and let $\JJ\subset (N_1,N_1+N_2]$ be a random set of size $q$
   determined by $X^{(1)}$. Then the variables $\max_{j\in \JJ} X_j$ and
   $\max_{N_1<j\le N_1+q} X_j$ are equidistributed.
\end{lem}

By applying the asymptotic expression \eqref{Emax_critical} to each term of the sum \eqref{greedy} and using
that the function $n\mapsto \log n$ is slowly varying we obtain the desired lower bound
\[
   \E \max_\sigma \sum_{i=1}^{n} X_{i\sigma(i)} \ge c\, H_*\, n\, \log n \, (1+o(1)) .
\]
Replacing maxima by minima in the greedy method and using \eqref{Emin_critical}
yields the remaining upper bound
\[
   \E \min_\sigma \sum_{i=1}^{n} X_{i\sigma(i)} \le   c \, \widetilde{H}_* \, n \, \log n \, (1+o(1)).
\]

This completes the proof of Theorem~\ref{t:critical} except for the postponed proof
of Lemma~\ref{l:random_sets}.
\end{proof} %% of Theorem {t:critical}
\medskip

%%%%%%%%%%%%%%%%%%%%%%%%%%%%%%%%%%%%%%%%%%%%%%%%%%%%%%%%

 \begin{proof}[Proof  of Lemma \ref{l:random_sets}]
Let
 \[
   S=S(X^{(1)}) := \sum_{j=1}^{N_1} X_j.
 \]
 Recall that the conditional distribution of $X^{(2)}$  w.r.t.  $X^{(1)}$ is $\MM_{m-S,N_2}$.
 This means that for all $x_1\in \N^{N_1}, x_2\in \N^{N_2}$ it is true that
 \[
    \P(X^{(2)}=x_2, X^{(1)}=x_1) =  \P(X^{(1)}=x_1) \ \MM_{m-S(x_1), N_2} (x_2).
 \]
For every fixed set $J\subset (N_1,N_1+N_2]$ of size $q$,
it holds that
\[
    \P(X^{(2)}=x_2, \JJ=J ) =  \sum_{s=0}^m \P(\JJ=J, S=s) \ \MM_{m-s, N_2} (x_2),
 \]
   by summing up over $x_1\in \JJ^{-1}(J)$.
  Now, for every non-negative integer $\mu$, by summing up over $x_2$ such that
 $\max_{j\in J} x_{2j} = \mu$, we obtain
 \[
    \P( \max_{j\in J} X_{j}=\mu   , \JJ=J ) =  \sum_{s=0}^m \P(\JJ=J, S=s) \
    \MM_{m-s, N_2} (x_2 : \max_{j\in J} x_{2j} = \mu ).
 \]
 The latter factor does not depend on a particular set $J$ due to exchangeability property of the multinomial law.
We thus may denote
 \[
    \MM_{m-s, N_2} (x_2 : \max_{j\in J} x_{2j} = \mu ) =: F(m-s,N_2,q,\mu)
 \]
 and obtain
 \[
    \P( \max_{j\in J} X_{j}=\mu   , \JJ=J ) =  \sum_{s=0}^m \P(\JJ=J, S=s) \ F(m-s,N_2,q,\mu).
 \]
 By summing up over all sets $J$ of size $q$
 we see that
 \[
    \P( \max_{j\in \JJ} X_{j}=\mu) =  \sum_{s=0}^m \P(S=s) \ F(m-s,N_2,q,\mu)
 \]
does not depend on the specific choice of $\JJ$, and the claim of lemma follows.
\end{proof} %% of Lemma
\medskip

%%%%%%%%%%%%%%%%%%%%%%%%%%%%%%%%%%%%%%%%%%%%%%%%

 \begin{proof}[Proof  of Theorem \ref{t:ER}]
We are going to use an old result by Erd{\H o}s and R\'enyi \cite{ER} about the existence of perfect matching
in a random bipartite graph. Let $G$ be a uniformly distributed
$n+n$ bipartite graph with $m=m(n)$ edges. If
\begin{equation} \label{ER}
    \lim \left( \frac{m}{n} -\log n\right) = \infty,
\end{equation}
then with probability tending to one, as $n\to \infty$, $G$ has a perfect matching.

In the matrix form, this result asserts the following. Let $Y=Y(n,m)=\{Y_{ij}\}_{1\le i,j\le n}$
be a uniformly distributed random $n \times n$ matrix with entries taking values in $\{0,1\}$ and satisfying
$\sum_{i,j=1}^n Y_{ij} =m$. If \eqref{ER} holds, then
\begin{equation} \label{ER_concl}
   \lim \P\left( \max_\sigma \sum_{i=1}^n Y_{i\sigma(i)} =n \right) =1.
\end{equation}

Let now $X=(X_{ij})$ be our matrix following the multinomial law  $\MM(m, n^{2})$.
Introduce the matrix $\tY$ by
\[
   \tY_{ij} := \begin{cases} 0, & X_{ij}>0,\\
   1,& X_{ij}=0.
   \end{cases}
\]
Note that
\[
   \P(\tY_{ij}=1) = \P ( X_{ij} =0)  = \left(1-p \right)^m = \exp\left( - m\, p\, (1+o(1)) \right).
\]
Let $S:= \sum_{i,j=1}^n \tY_{ij}$  be the number of empty cells in our matrix $X$. Observe that, conditioned on $S$,  the matrix $\tY$ has the same distribution as $Y(n,S)$.
Taking into account that the probability in \eqref{ER_concl} is non-decreasing as a function
of $m$, we have for every positive integer $M$
\begin{eqnarray} \nonumber
 && \P\left( \min_\sigma \sum_{i=1}^n X_{i\sigma(i)} =0 \right)
  = \P\left( \max_\sigma \sum_{i=1}^n \tY_{i\sigma(i)} =n \right)
\\  \label{product}
  &\ge& \P(S\ge M) \ \P \left(   \max_\sigma \sum_{i=1}^n Y(n,M)_{i\sigma(i)} =n \right).
\end{eqnarray}
We choose $M=n^{\beta}$ with $\beta\in(1,2-c)$ and show that both probabilities in the latter product tend to one as $n\to\infty$.

For the first one, using \eqref{alpha}, we have
\[
  \E S = n^2 \E\tY_{11} = n^2 \exp\left( - m\, p\, (1+o(1)) \right)
  \ge n^{2-c(1+o(1))} .
\]
Furthermore, since the variables $\tY_{ij}$ are negatively correlated, we have
\[
  \Var S  \le n^2\,  \Var\tY_{11}  \le n^2\,  \E\tY_{11} = \E S .
\]
Finally, using $\beta < 2-c$, by Chebyshev inequality,
\begin{eqnarray*}
  \P(S\le n^\beta) &\le& \P(|S-\E S|\ge \E S-n^\beta) = \P(|S-\E S|\ge \E S (1+o(1)))
\\
   &\le& \frac{\Var S}{(\E S)^2 (1+o(1))} \le \frac{\E S}{(\E S)^2 (1+o(1))} \to 0.
\end{eqnarray*}
 On the other hand, since $\beta>1$, the assumption \eqref{ER} with $m:=M=n^\beta$ is true. Therefore, the second probability
 in the product \eqref{product} tends to one by Erd{\H o}s--R\'enyi result. We obtain  from \eqref{product} that
 \[
   \lim \P\left( \min_\sigma \sum_{i=1}^n X_{i\sigma(i)} =0 \right) =1,
 \]
 which is the desired claim.
\end{proof}
\medskip

%%%%%%%%%%%%%%%%%%%%%%%%%%%%%%%%%%%%%%%%%%%%%%%

 \begin{proof}[Proof  of Theorem \ref{t:Poisson}]
The proof goes along the same lines as the one of Theorem~\ref{t:critical}.
Instead of the key relation \eqref{Emax_critical}, we prove the following claim.
Let $(X_j)$ be negatively associated random variables following Bernoulli law $\BB(m,p)$.
Then under assumptions \eqref{Poisson} and \eqref{delta} it is true that
\begin{equation}  \label{Emax_Poisson}
  \E \max_{1\le j\le n}  X_j
  = \frac{\log n}{\log\left(\frac{\log n}{mp}\right)}\, (1+o(1)),
  \qquad \textrm{as } n\to\infty.
\end{equation}

{\bf Upper bound.}
For the upper bound in \eqref{Emax_Poisson} that we are going to prove now, no lower bound on $mp$ is needed;
we only use \eqref{Poisson}.

Let $\beta>1$, $y:=\tfrac{\beta \log n}{mp}$, $r:=\frac{y}{\log y}$. Notice that under \eqref{Poisson}
we have $y,r\to \infty$. Next, for a Binomisal $\BB(m,p)$ random variable $X$ and
for every $v>0$ it is true that
\begin{eqnarray*}
 \P \left(X\ge \frac{\beta\log n}{\log\left(\frac{\log n}{mp}\right)} +v \right)
 &\le&
  \P \left(X\ge \frac{ \beta \log n}{\log\left(\frac{\beta \log n}{mp}\right)} +v \right)
\\
  &=&
   \P \left(X\ge \frac{\beta\,
   \frac {\log n}{mp}} {\log\left(\frac{\beta \log n}{mp}\right)} \ mp  +v\right)
\\
  &=&
  \P\left( X\ge \frac{y}{\log y} \ mp +v \right)
  =
   \P\left( X\ge r \, mp +v \right).
\end{eqnarray*}
In the next calculation we use the Poisson version of the bound for exponential moment
\[
   \E \exp(\gamma X) \le \exp(mp (e^\gamma-1))
\]
that immediately follows from the exact formula \eqref{Eexp}.
By applying Chebyshev inequality with Poisson-optimal parameter $\gamma=\log r$ we obtain
\begin{eqnarray*}
    \P\left( X\ge rmp+v \right)
    &\le&  \E \exp(\gamma X) \exp(-\gamma (rmp+v))
\\
     &\le&  \exp( -mp (\gamma r-e^\gamma+1) -\gamma v)
\\
      &=&  \exp( - m p (r\log r - r+1)- \gamma v).
\end{eqnarray*}
Since $r\to\infty$, we have
\[
   r \log r - r+1\sim r\log r \sim y = \frac{\beta \log n}{mp}.
\]
It follows that
\begin{eqnarray*}
    \P\left( X \ge rmp +v \right)
    &\le&  \exp( - \beta \log n (1+o(1)) -\gamma v)
\\
   &=& n^{- \beta (1+o(1))} \exp(-\gamma v).
\end{eqnarray*}
and
\[
    \P\left(\max_{1\le j\le n} X_j \ge rmp +v \right)
    \le n\   \P\left( X \ge rmp +v \right)
    \le n^{-(\beta-1) (1+o(1))} \exp(-\gamma v).
\]
Hence,
\begin{eqnarray*}
   \E \max_{1\le j\le n} X_j
   &\le&
   rmp +  n^{-(\beta-1)(1+o(1))} \int_0^\infty \exp(-\gamma v)dv
\\
   &=& rmp +  n^{-(\beta-1)(1+o(1))} \gamma^{-1}.
\end{eqnarray*}
Note that
\[
  rmp\gamma = r\log r \, mp \sim  y\, mp = \beta \log n \to \infty,
\]
hence we conclude that  $n^{-(1-\beta)(1+o(1))} \gamma^{-1}$ is negligible compared to $rmp$,
thus
\begin{eqnarray*}
  \E \max_{1\le j\le n} X_j \le  rmp (1+o(1))
  \sim \frac{\beta\log n}{\log\left(\frac{\log n}{mp}\right)}
\end{eqnarray*}
and the required upper bound follows by letting $\beta \searrow 1$.
\medskip

{\bf Lower bound.}
Let $\beta\in (0,1)$, $y:=\tfrac{\beta \log n}{mp}$, $r:=\tfrac{y}{\log y}$, and
\[
  k:= r mp = \frac{y}{\log y}\, mp = \frac {\beta \log n}{\log y} .
\]
Assumption \eqref{Poisson} yields $y\to\infty$, $k=o(\log n)$, $e^k=n^{o(1)}$,
$e^{mp}=n^{o(1)}$.

On the other hand, under assumption \eqref{delta}
we have $|\log(mp)|\ll \log n$, which yields $\log y \ll \log n$, hence $k\to \infty$.

Therefore, by using Poissonian approximation, we obtain
\begin{eqnarray*}
 \P(X\ge k) &\ge&  \P(X=k)
 \sim e^{-mp} \frac {(mp)^k}{k!}
 \sim  e^{-mp} \ e^k\, (2\pi k)^{-1/2}  \left(\frac{mp}{k}\right)^k
\\
  &=&  n^{o(1)}\, r^{-k}
  =  n^{o(1)} \,r^{-rmp}
  =  n^{o(1)} \, \exp(-r \log r\ mp)
\\  &=& n^{o(1)}\, \exp (-y (1+o(1))\, mp)
     = n^{-\beta+o(1)}.
\end{eqnarray*}
By repeating the arguments from  \eqref{step3_a}, \eqref{step3_b}, and
\eqref{step4} we obtain
\[
    \E \max_{1\le j\le n} X_j  \ge k (1+o(1))
    = \frac{y}{\log y}\ mp \, (1+o(1))
    = \frac{\beta \log n}{\log y}\, (1+o(1))
\]
and letting $\beta\nearrow 1$ provides the required lower bound in \eqref{Emax_Poisson}.

Once \eqref{Emax_Poisson} is proved, the proof of Theorem~\ref{t:Poisson} is completed by the same
simple arguments (including the greedy method) as that of Theorem~\ref{t:critical}.
\end{proof}
\medskip

%%%%%%%%%%%%%%%%%%%%%%%%%%%%%%%%%%%%%%%%%%%%%%%%%%%%%%%%%%%%%%%%%%%%%%

 \begin{proof}[Proof  of Theorem \ref{t:Rsparse}]

{\bf Upper bound}. We have
\begin{eqnarray*}
    &&     \E \max_{1\le j\le n} X_j
\\
     &=&   \E \left[\max_{1\le j\le n} X_j \ed{ \max_{1\le j\le n} X_j \le k } \right]
     + \E \left[\max_{1\le j\le n} X_j \ed{ \max_{1\le j\le n} X_j > k } \right]
\\
       &\le&  k + \sum_{j=1}^n \E \left[ X_j \ed{X_j > k} \right]
        =  k +  n \, \E \left[ X_1 \ed{X_1 > k} \right].
\end{eqnarray*}
Furthermore,  since the law of $X_1$ is $\BB(m,p)$, it is true that
\[
   \P(X_1=\ell) = \frac{m!}{(m-\ell)!} \, \frac{p^\ell}{\ell!} \, (1-p)^{m-\ell}
   \le \frac{m^\ell p^\ell}{\ell!}, \qquad 0\le \ell \le m.
\]
Hence,
\begin{eqnarray*}
    \E \left[ X_1 \ed{X_1 > k} \right] &\le&  \sum_{\ell=k+1}^\infty \frac{(mp)^\ell}{(\ell-1)!}
     =  \sum_{q=0}^\infty \frac{(mp)^{k+1+q}}{(k+q)!}
\\
    &\le& (mp)^{k+1} \exp(mp) = (mp)^{k+1} (1+o(1)).
\end{eqnarray*}
Therefore,
\begin{equation} \label{Emax_a_upper}
   \E \max_{1\le j\le n} X_j \le k + c^{k+1} n^{1-a(k+1)}\, (1+o(1)) = k+o(1),
\end{equation}
where we used the lower bound in \eqref{bounds_a} at the last step.
\medskip

Turning to the lower bound, for every positive integer $v$  in the {\it  independent case},
we have
\begin{eqnarray} \nonumber
    \P\left( \max_{1\le j\le v} X_j < k \right)
    &=&  \P\left( X_1 < k \right)^{v}
     = \left(1- \P\left( X_1 \ge k \right)  \right)^{v}
\\  \nonumber
   &\le&   \left(1- \P\left( X_1 = k \right)  \right)^{v}
\\   \nonumber
   &=&  \exp\{ - v \ \P\left( X_1 = k \right)\,  (1+o(1))\}
\\ \label{bound_v}
      &=& \exp\left\{ - v \  \frac{c^k n^{-a k}}{k!} \, (1+o(1))\right\}.
\end{eqnarray}
Let us fix some small $\delta\in (0,1)$. By letting $v=[\delta n]$ and  using the upper bound in \eqref{bounds_a} we obtain
\[
      \P\left( \max_{1\le j\le [\delta n] } X_j < k \right)  \to 0.
\]
It follows that
\[
    \E \max_{1\le j\le [\delta n]} X_j
    \ge k  \ \P\left( \max_{1\le j\le [\delta n]} X_j \ge k \right)
    = k \, (1+o(1)),
\]

By using negative association argument \eqref{step4}, we also
obtain
\begin{equation} \label{Emax_a_lower}
    \E \max_{1\le j\le [\delta n]} X_j  \ge  k \, (1+o(1))
\end{equation}
in the multinomial setting.

Finally, by using \eqref{Emax_a_upper} and the greedy method based on \eqref{Emax_a_lower}, we conclude that
in the regular case \eqref{regular_a} for the assignment process it is true that
\[
    \E \max_\sigma \sum_{i=1}^n X_{i\sigma(i)} =  k\, n \, (1+o(1)). \qedhere
\]
\end{proof}
\medskip

%%%%%%%%%%%%%%%%%%%%%%%%%%%%%%%%%%%%%%%%%%%%%%%

 \begin{proof}[Proof  of Theorem \ref{Vsparse}]
The upper bound
\[
   \max_\sigma \sum_{i=1}^n X_{i\sigma(i)} \le m
\]
is trivial; it remains to prove the lower bound.

Let us denote $(u_i,v_i)_{1\le i\le m}$ the coordinates of the particles thrown on the square table.
All $u_i$ and all $v_i$ are i.i.d. random variables uniformly distributed on integers $[1..n]$.
Let  $U_0=V_0=\emptyset$,
\[
   U_k := \left\{ u_i, 1\le i\le k\right\}, \quad  V_k := \left\{ v_i, 1\le i\le k\right\},
   \qquad 1\le k\le m,
\]
and introduce the events
\[
   A_k := \left\{ u_{k} \not \in U_{k-1},  v_{k} \not \in V_{k-1} \right\}, \quad 1\le k\le m.
\]
It is obvious that for each $k$
\[
  \P(A_k) \ge 1-\frac{2m}{n},
\]
hence by $m \ll n$
\[
   \E \left(\sum_{k=1}^m  \ed{A_k} \right) \ge m \left( 1-\frac{2m}{n} \right) = m \, (1+o(1)).
\]
On the other hand, we have
\begin{equation}\label{bound_Ak}
    \max_\sigma \sum_{i=1}^n X_{i\sigma(i)} \ge \sum_{k=1}^m  \ed{A_k},
\end{equation}
which entails the desired
\[
    \E \max_\sigma \sum_{i=1}^n X_{i\sigma(i)} \ge  m \, (1+o(1)). \qedhere
\]
\end{proof}
\medskip

{\bf Acknowledgements.} The work of M.\ Lifshits was supported by RSF grant
21-11-00047.
G.\ Mordant gratefully acknowledges the support of the DFG within SFB 1456.
%%%%%%%%%%%%%%%%%%%%%%%%%%%%%%%%%%%%%%%%%%%%%%%%%%%%%%%%%%%%%%%%%%%%%%%%%%%%%%%%%%%%%%%%%%%%%

\end{document}